\documentclass[notitlepage]{amsart}
\usepackage{amsmath}%
\usepackage{amsfonts}%
\usepackage{amssymb}%
\usepackage{stmaryrd}
\usepackage[pdftex]{graphicx,color}
\input xy
\xyoption{all}

\newtheorem{theorem}{Theorem}
\theoremstyle{plain}

\newtheorem{corollary}{Corollary}
\newtheorem{lemma}{Lemma}
\newtheorem{proposition}{Proposition}

\theoremstyle{definition}
\newtheorem{definition}{Definition}
\theoremstyle{remark}
\newtheorem{example}{Example}

\newtheorem{remark}{Remark}

\numberwithin{equation}{section}

\newcommand{\lb}{\llbracket}
\newcommand{\rb}{\rrbracket}
\newcommand{\rbp}{{\rb_\phi}}
\newcommand{\AV}{\mathbb{A}}
\newcommand{\TTM}{\mathbb{T}M}

\newcommand{\id}{\mathsf{id}}
\newcommand{\Lied}{\mathcal{L}}
\renewcommand{\varpi}{{\tilde\pi}}
\newcommand{\jpi}{j}
\renewcommand{\varrho}{{\tilde\rho}}

\newcommand{\GL}{\mathbf{GL}}

\DeclareMathOperator{\Ann}{Ann}
\DeclareMathOperator{\Hom}{Hom}
\DeclareMathOperator{\End}{End}

\providecommand{\abs}[1]{\lvert#1\rvert}

\begin{document}
\title[$AV$-Courant Algebroids]{$AV$-Courant Algebroids and Generalized CR Structures}
\author{David Li-Bland}


\begin{abstract}
We construct a generalization of Courant algebroids which are classified by the third cohomology group $H^3(A,V)$, where $A$ is a Lie Algebroid, and $V$ is an $A$-module. We see that both Courant algebroids and $\mathcal{E}^1(M)$ structures are examples of them. Finally we introduce generalized CR structures on a manifold, which are a generalization of generalized complex structures, and show that every CR structure and contact structure is an example of a generalized CR structure. 
\end{abstract}
\maketitle

\section{Introduction}

Courant algebroids and the Dirac structures associated to them were first introduced by Courant and Weinstein (see \cite{CW88} and \cite{C90}) to provide a unifying framework for studying such objects as Poisson and symplectic manifolds. A\"{i}ssa Wade later introduced the related $\mathcal{E}^1(M)$-Dirac structures in \cite{W00} to describe Jacobi structures.

In \cite{H03}, Hitchin defined generalized complex structures which are further described by Gualtieri in \cite{G07}. Generalized complex structures unify both symplectic and complex structures, interpolating between the two, and have appeared in the context of string theory \cite{LMTZ05}. In \cite{IW05} Iglesias and Wade describe generalized contact structures, an odd-dimensional analog to generalized complex structures, using the language of $\mathcal{E}^1(M)$-Dirac structures.

In this paper, we shall define $AV$-Courant Algebroids, a generalization of Courant algebroids which also allows one to describe $\mathcal{E}^1(M)$-Dirac structures. We will show that these have a classification similar to \u{S}evera's classification of exact Courant algebroids in \cite{SW01}.

To be more explicit, let $M$ be a smooth manifold, $A\to M$ be a Lie algebroid with anchor map $a:A\to TM$,  and $V\to M$ a vector bundle which is an $A$-module. If we endow $V$ with the structure of a trivial Lie algebroid (that is: trivial bracket and anchor), then it is well know that the extensions of $A$ by $V$ are a geometric realization of $H^2(A,V)$ (see \cite{M05}, for instance). In this paper, we introduce $AV$-Courant algebroids and describe how they are a geometric realization of $H^3(A,V)$.

We then go on to show how to simplify the structure of certain $AV$-Courant algebroids by pulling them back to certain principal bundles. Indeed, in the most interesting cases, the pullbacks will simply be exact Courant algebroids.

We then introduce $AV$-Dirac structures, a special class of subbundles of an $AV$-Courant algebroid which generalize Dirac structures. Finally, we will introduce a special class of $AV$-Dirac structures, called generalized CR structures, which allow us to describe any complex, symplectic, CR or contact structure on a manifold, as well as many interpolations of those structures. Poisson and Jacobi structures are central to this description; and while Jacobi brackets do not generally satisfy the leibniz rule, as a consequence of the ideas in this paper we show how one can think of a Jacobi structure on a manifold as a Lie bracket on the sections of a certain line bundle which does satisfy the Leibniz rule.

It is important to note that there are other constructions related to $AV$-Courant algebroids. For instance, recently Z. Chen, Z. Liu and Y.-H. Sheng introduced the notion of $E$-Courant algebroids \cite{CLS08E} in order to unify the concepts of omni-Lie algebroids (introduced in \cite{CL07Om}, see also \cite{CLS08Om}) and generalized Courant algebroids or Courant-Jacobi algebroids (introduced in \cite{NC04} and \cite{GM03} respectively; they are equivalent concepts see \cite{NC04}). The key property that both $E$-Courant algebroids and $AV$-Courant algebroids share is that they replace the $\mathbb{R}$-valued bilinear form of Courant algebroids with one taking values in an arbitrary vector bundle ($E$ or $V$ respectively). Nevertheless, while there is some overlap between $E$-Courant algebroids and $AV$-Courant algebroids in terms of examples, these constructions are not equivalent. Moreover, this paper is distinguished from \cite{CLS08E} by having the definition of generalized CR manifolds as one of its main goals.

Meanwhile generalized CRF structures, introduced and studied in great detail by Izu Vaisman in \cite{V08}, and generalized CR structures describe similar objects. To summarize, a complex structure on a manifold $M$ is a subbundle $H\subset TM\otimes\mathbb{C}$ such that 
\begin{equation}\label{complex}H\oplus\bar H=TM\otimes\mathbb{C}\end{equation}
 and $[H,H]\subset H$. The definition of a CR structure simply relaxes (\ref{complex}) to $H\cap\bar H=0$. On the other hand, the definition of a generalized complex structure replaces $TM$ with the standard Courant algebroid $\TTM=T^*M\oplus TM$ in the definition of a complex structure, and in addition, requires $H\subset \TTM\otimes\mathbb{C}$ to be isotropic.

 The definition of a generalized CRF structure parallels the definition of a generalized complex structure, but relaxes the requirement that $H\oplus\bar H=\TTM\otimes\mathbb{C}$ to $H\cap\bar H=0$. Among numerous interesting examples of generalized CRF structures are normal contact structures and normalized CR structures (namely those CR structures $H\subset TM\otimes\mathbb{C}$ for which there is a splitting $TM\otimes\mathbb{C}=H\oplus\bar H\oplus Q_c$ and $[H,Q_c]\subset H\oplus Q_c$).

 Generalized CR structures differ from generalized CRF structures in multiple ways, in particular they replace the standard Courant algebroid with an $AV$-Courant algebroid $\AV$, and furthermore they take a different approach to describe contact and CR structures, using only maximal isotropic subbundles but allowing $H\cap\bar H$ to contain `infinitesimal' elements.

\vskip.2in \noindent {\bf Acknowledgements.} We would like to thank
Eckhard Meinrenken for all his helpful suggestions and discussions, his patience and his encouragement. We would like to thank A\"{i}ssa Wade and Henrique Bursztyn for their encouragement and suggestions. D.L.-B. was supported by an NSERC CGS-D Grant.
 
\section{$AV$-Courant Algebroids}
Let $M$ be a smooth manifold, $A\to M$ a Lie algebroid,  and $V\to M$ a vector bundle which is an $A$-module, that is, there is a $C^\infty(M)$-linear Lie algebra homomorphism 
\begin{equation}\label{AModule}\Lied_{\cdot}:\Gamma(A)\to \End(\Gamma(V))\end{equation}
satisfying the Leibniz rule. (See \cite{M05} for more details.)

For any $A$-module $V$, the sections of $V\otimes\wedge^* A^*$ have the structure of a graded right $\wedge^*\Gamma( A^*)$-module, and there are several important derivations of its module structure which we shall use throughout this paper. The first is the interior product with a section $X\in\Gamma(A)$, $$\iota_X:\Gamma (V\otimes \wedge^i  A^*)\to\Gamma (V\otimes\wedge^{i-1}A^*),$$ a derivation of degree -1.

The second is the Lie derivative, a derivation of degree 0, defined to be the unique derivation of $V\otimes\wedge^* A^*$ whose restriction to $V$ is given by (\ref{AModule}), and such that the graded commutator with $\iota_\cdot$ satisfies
$$[\Lied_X,\iota_Y]=\iota_{[X,Y]}.$$
Finally the differential $d$, a derivation of degree 1, is defined inductively by the graded commutator $\Lied_X=[d,\iota_X]$ (for all $X\in\Gamma(A)$).

It is easy to check that $d^2=0$, and the cohomology groups of the complex $(\Gamma (V\otimes \wedge^\bullet  A^*),d)$ are denoted $H^\bullet(A,V)$.

\subsection{Definition of $AV$-Courant Algebroids}
Let $A$ be a Lie algebroid, and $V$ an $A$-module.
\begin{definition}[$AV$-Courant Algebroid]\label{defnextend}
Let $\AV$ be a vector bundle over $M$, with a $V$-valued symmetric bilinear form $\langle\cdot,\cdot\rangle$ on the fibres of $\AV$, and a bracket $\lb,\rb$ on sections of $\AV$. Suppose further that there is a short exact sequence of bundle maps 
\begin{equation}\label{exactsequence}0\to V\otimes A^*\xrightarrow{ \jpi}\AV\xrightarrow{\pi}A\to 0.\end{equation}
such that for any $e\in\Gamma(\AV)$ and $\xi\in\Gamma(V\otimes A^*)$,
\begin{equation}\label{bracketcond}\langle e, \jpi(\xi)\rangle=\iota_{\pi(e)}\xi.\end{equation}
 The bundle $\AV$ with these structures is called an $AV$-Courant algebroid if, for $f\in C^\infty(M)$, and $e,e_i\in\Gamma(\AV)$, the following axioms are satisfied:

\begin{enumerate}
\renewcommand{\labelenumi}{AV-\arabic{enumi}}
\item $\lb e_1,\lb e_2, e_3 \rb\rb=\lb\lb e_1,e_2\rb,e_3\rb+\lb e_2,\lb e_1,e_3\rb\rb$\label{AX4}

\item $\pi(\lb e_1,e_2\rb)=[\pi(e_1),\pi(e_2)]$\label{AX5}
\item $\lb e, e\rb = \frac{1}{2}D\langle e,e\rangle$ where $D=\jpi\circ d$\label{AX6}
\item $\Lied_{\pi(e_1)}\langle e_2,e_3\rangle = \langle\lb e_1,e_2\rb,e_3\rangle+\langle e_2,\lb e_1,e_3\rb\rangle $\label{AX7}
\end{enumerate}

we will often refer to $\lb\cdot,\cdot\rb$ as the Courant bracket.

\end{definition}

\begin{remark}\label{brktDer}
Axioms (AV-\ref{AX4}) and (AV-\ref{AX7}) state that $\lb e,\cdot\rb$ is a derivation of both the Courant bracket and the bilinear form, while Axiom (AV-\ref{AX5}) describes the relation of the Courant bracket to the Lie algebroid bracket of $A$. One should interpret Axiom (AV-\ref{AX6}) as saying that the failure of $\lb\cdot,\cdot\rb$ to be skew symmetric is only an `infinitesimal' $D(\cdot)$.

The bracket is also derivation of $\AV$ as a $C^\infty(M)$-module in the sense that
 $$\lb e_1,f e_2\rb=f\lb e_1,e_2\rb + a\circ\pi(e_1)(f)\cdot e_2$$
for any $e_1,e_2\in\Gamma(\AV)$ and $f\in C^\infty(M)$. In fact if $e_3\in\Gamma(\AV)$,
$$\begin{array}{rl}
& \langle a\circ\pi(e_1)(f)\cdot e_2 +f\lb e_1,e_2\rb -\lb e_1,f e_2\rb, e_3\rangle\\
\text{(by (AV-\ref{AX7})) }=& \langle a\circ\pi(e_1)(f)\cdot e_2 +f\lb e_1,e_2\rb,e_3\rangle -\pi(e_1)\langle f e_2, e_3\rangle+\langle f e_2,\lb e_1,e_3\rb\rangle\\
=&a\circ\pi(e_1)(f)\langle e_2,e_3\rangle -\pi(e_1)\langle f e_2, e_3\rangle+f(\langle\lb e_1,e_2\rb,e_3\rangle+\langle  e_2,\lb e_1,e_3\rb\rangle)\\

\text{(by (AV-\ref{AX7})) }=&a\circ\pi(e_1)(f)\langle e_2,e_3\rangle -\pi(e_1)\langle f e_2, e_3\rangle+f\pi(e_1)\langle e_2,e_3\rangle\\
=&0,\\
\end{array}$$
where the last equality follows from the fact that $V$ is an $A$ module. Since this holds for all $e_3\in\Gamma(\AV)$, and $\langle\cdot,\cdot\rangle$ is non-degenerate, the statement follows.

\end{remark}

\begin{remark}

One notices that (\ref{bracketcond}) and exactness of (\ref{exactsequence}) implies that the map 
\begin{equation}\label{altcond}e\to \langle e,\cdot\rangle:\AV\to V\otimes \AV^*\end{equation} is an injection.
Consequently, if $V$ is a line bundle, it follows that $\AV\simeq V\otimes \AV^*$, and $\jpi$ must be the composition
$$\jpi: V\otimes A^*\xrightarrow{\id\otimes\pi^*}V\otimes\AV^*\simeq\AV.$$
\end{remark}

\subsection{Splitting}

We call $\phi:A\to \AV$ an isotropic splitting, if it splits the exact sequence (\ref{exactsequence}) and $\phi(A)$ is an isotropic subspace of $\AV$ with respect to the inner product.

\begin{remark}
Such splittings exist. In fact we may choose a splitting $\lambda:A\to \AV$, which is not necessarily isotropic.

Then we have a map $\gamma:A\to V\otimes A^*$ given by the composition
$$\gamma:A\xrightarrow{\lambda}\AV\xrightarrow{e\to\langle e,\cdot\rangle }V\otimes\AV^*\xrightarrow{\id\otimes\lambda^*} V\otimes A^*.$$
We let $\phi=\lambda-\frac{1}{2}\jpi\circ\gamma$. It is easy to check that $\phi$ is an isotropic splitting.

\end{remark}

If $\phi:A\to\AV$ is an isotropic splitting, then we have an isomorphism $\phi\oplus\jpi:A\oplus(V\otimes A^*)\to\AV$.

\begin{proposition}\label{mainprop}
Let $\phi:A\to\AV$ be an isotropic splitting. Then under the above isomorphism, the bracket on $A\oplus(V\otimes A^*)$ is given by
\begin{equation}\label{bracket}\lb X+\xi,Y+\eta\rbp=[X,Y]+\Lied_X\eta - \iota_Y d\xi + \iota_X\iota_Y H_\phi,\end{equation}
where $X,Y\in \Gamma(A)$, $\xi,\eta\in \Gamma(V\otimes A^*)$, and $H_\phi\in\Gamma(V\otimes\wedge^3 A^*)$, with $ d H_\phi = 0$.

Furthermore, if $\psi:A\to\AV$ is a different choice of isotropic splitting, then $\psi(X)=\phi(X)+\jpi(\iota_X\beta)$, and $H_\psi=H_\phi- d\beta$, where $\beta\in\Gamma(V\otimes\wedge^2 A^*)$.
\end{proposition}
The proof is relegated to the appendix, since it is parallel to the proof for ordinary Courant algebroids (see \cite{BC05} or \cite{SW01}).

\begin{theorem}\label{classifyingthm}
Let $A$ be a Lie algebroid and $V$ an $A$-module. Then the isomorphism classes of $AV$-Courant algebroids are in bijective correspondence with $H^3(A,V)$.
\end{theorem}
\begin{proof}
If $H\in \Gamma(V\otimes\wedge^3 A^*)$, and $d H = 0$, then let $\AV = A\oplus(V\otimes A^*)$. We define $\langle\cdot,\cdot\rangle$ by
\begin{equation}\label{pairing}\langle X+\xi,Y+\eta\rangle =\iota_X\eta+\iota_Y\xi,\end{equation}
where $\xi,\eta\in\Gamma(V\otimes A^*)$ and $X,Y\in\Gamma(A)$. We define the bracket to be given by Equation~(\ref{bracket}). It is not difficult to check that this satisfies the axioms of an $AV$-Courant algebroid.

Conversely, by the above proposition, every $AV$-Courant algebroid defines a unique element of $H^3(A,V)$.
\end{proof}

\section{Examples}

\begin{example}
Let $M$ be a point, then  a Lie algebroid $A$ is simply a Lie algebra, and an $A$-module $V$ is a finite dimensional representation of $A$ as a Lie algebra. $H^i(A,V)$ is simply the $V$-valued Lie algebra cohomology, and $H^3(A,V)$ classifies the $AV$-Courant algebroids over a point. Note that an $AV$-Courant algebroid over a point is a Lie algebra if and only if $V$ is a trivial $A$-representation.
\end{example}

\begin{example}[Exact Courant Algebroids]
If we let  $A\simeq TM$ and $V=M\times \mathbb{R}$ be the trivial line bundle with a trivial $TM$-module structure, then we have the class of exact Courant algebroids (see \cite{CW88} or \cite{C90}) on $M$, 
$$\xymatrix@1{0\ar[r]&T^*M\ar[r]^{\pi^*} &\AV\ar[r]^\pi & TM\ar[r]&0}.$$
Theorem~\ref{classifyingthm} then corresponds to \u{S}evera's classification of exact Courant algebroids.
\end{example}

\begin{example}[$\mathcal{E}^1(M)$ Structures]\label{e1mcourant}
The bundle $\mathcal{E}^1(M)$ was introduced by A. Wade in \cite{W00}, and is uniquely associated to a given manifold $M$. Within the context of our paper, it is easiest to define $\mathcal{E}^1(M)$ by using the language of $AV$-Courant algebroids:

We let $A=TM\oplus L$, where $L\simeq \mathbb{R}$ is spanned by the abstract symbol $\frac{\partial}{\partial t}$. The bracket is given by
$$[X\oplus f\frac{\partial}{\partial t},Y\oplus g\frac{\partial}{\partial t}]_A=[X,Y]_{TM}\oplus(X(g)-Y(f))\frac{\partial}{\partial t},$$
where $X,Y\in\mathcal{X}(M)$, and $f,g\in C^\infty(M)$.

Let $V$ be the trivial line bundle spanned by the abstract symbol $e^t$, so that $\Gamma(V)=\{e^t h|h\in C^\infty(M)\}$. $V$ has an $A$-module structure (as suggested by the choice of symbols) given by
$$({X\oplus f\frac{\partial}{\partial t}})(e^t h)= e^t(X(h)+f h).$$
We let $\AV:= (TM\oplus L)\oplus(V\otimes(T^*M\oplus L^*))$,  and define a bracket on sections by Equation~(\ref{bracket}). It is clear that this data defines an $AV$-Courant algebroid on $M$. If we set $H=0$ in Equation~(\ref{bracket}), then the pair $(\AV,\lb\cdot,\cdot\rb)$ associated to $M$ is the $\mathcal{E}^1(M)$-Structure, as introduced by A. Wade in \cite{W00}.
\end{example}

\begin{example}[Equivariant $AV$-Courant Algebroids on Principal Bundles]\label{principal}
Let $\nu:P\to M$ be a $G$-principal bundle. Suppose that $A$ is a Lie algebroid over $P$ and $V$ is an $A$-module; and that there is an $AV$-Courant algebroid on $P$,
$$0\to V\otimes A^*\to\AV\to A\to 0.$$
If the action of $G$ on $P$ lifts to an action by bundle maps on $V$, $A$ and $\AV$, such that all the structures involved are $G$-equivariant. Then the quotient,
$$0\to (V\otimes A^*)/G\to\AV/G\to A/G\to 0,$$
is an $A\!/\!G\;V\!/\!G$-Courant algebroid.
\end{example}

\begin{example}\label{TPWcourant}
Let $\nu:P\to M$ be a $G$-principal bundle, and $W$ a $k$-dimensional vector space possessing a linear action of $G$. We regard $W$ as a trivial bundle over $P$, and we consider the bundle $\mathbb{T}:=TP\oplus(W\otimes T^*P)$, endowed with a $W$-valued symmetric bilinear form given by Equation~(\ref{pairing}). We also define a bracket on sections of $\mathbb{T}$ by Equation~(\ref{bracket}) where $H\in\Omega^3(P,W)^G$ is closed, then
$$0\to W\otimes T^*P\xrightarrow{\jpi} \mathbb{T} \xrightarrow{\pi} TP\to 0$$
is an equivariant $T\!P\:W$-Courant algebroid on $P$ (where $\jpi$ and $\pi$ are the obvious inclusion and projection). Thus (as in Example~\ref{principal}), we have an $AV$-Courant algebroid on $P/G$, where $A=TP/G$ is the Atiyah algebroid, and $V=P\times_G W$

Note, if $W$ is 1-dimensional, then the $T\!P\:W$-Courant algebroid given above is simply an exact Courant algebroid.

As it turns out, this is quite a general example. Indeed if $A$ is a transitive Lie algebroid, then locally all $AV$-Courant algebroids result from such a construction (See Section~\ref{princconv}).
\end{example}

\begin{remark}
In the above example, one could replace $P\times W$ with any flat bundle.
\end{remark}

\begin{example}\label{e1mprincipal}
As a special case of Example~\ref{TPWcourant}, if we take $G=\mathbb{R}$, then $P=M\times \mathbb{R}$ is a $G$-principal bundle where the action is translation. We let $W$ be the trivial line bundle over $P$ and let $t\in G$ act on $W$ by scaling by $e^t$.

To describe the $G$-action explicitly, we make the identification $\Gamma(W)=C^\infty(M\times\mathbb{R})$; and then the action is given by
\begin{equation}\label{e1maction}
\begin{array}{rcl}
\mathbb{R}\times C^\infty(M\times \mathbb{R})&\to&C^\infty(M\times \mathbb{R})\\
(t,f(x,s))&\to& e^{-t}f(x,s+t)\\
\end{array}
\end{equation}
 The quotient of the $T\!P\:W$-Courant algebroid on $P$ with $H=0$ under this action is precisely the $\mathcal{E}^1(M)$-Structure on $M=P/\mathbb{R}$.
\end{example}

\begin{example}
If $A$ is a Lie algebroid over $M$, $V$ is an $A$-module, and $\AV$ is an $AV$-Courant algebroid on the manifold $M$, and if $F\subset M$ is a leaf of the singular foliation defined by $a(A)$, then $i^*\AV$ is an $i^*\!A\:i^*\!V$-Courant algebroid on $F$, where $i:F\to M$ is the inclusion.
\end{example}

\section{$AV$-Dirac Structures}

\begin{definition}[$AV$-Dirac Structure]
Let $M$ be a manifold, $A\to M$ be a Lie algebroid over $M$, $V\to M$ an $A$-module, and $\AV$ an $AV$-Courant algebroid. Suppose that $L\subset \AV$ is a subbundle, since $\AV$ has a non-degenerate inner product, we can define $L^\perp=\{v\in\AV\mid\thickspace\langle v,u\rangle =0\thickspace\forall u\in L\}$.

We call $L$ an \textit{almost $AV$-Dirac structure} if $L^\perp=L$. An \textit{$AV$-Dirac} structure is an almost $AV$-Dirac structure, $L\subset\AV$ which is involutive with respect to the bracket $\lb,\rb$.
\end{definition}

\begin{remark}
If $L\subset\AV$ is an $AV$-Dirac structure, then $\lb e, e\rb = \frac{1}{2}D\langle e,e\rangle =0$ for any section $e\in\Gamma(L)$, so $\lb,\rb$ is skew-symmetric when restricted to $L$, and then by the other properties of the bracket, it follows that $a\circ\pi:L\to TM$ is a Lie algebroid, and $\pi:L\to A$ is a Lie algebroid morphism.
\end{remark}

\begin{example}[Invariant Dirac Structure on a Principal Bundle]\label{principalDirac}
Using the notation of Example~\ref{principal}, suppose that the $A/G\;V/G$-Courant algebroid $\AV/G$ on $M$ is the quotient of a $AV$-Courant algebroid $\AV$ on $P$. If $L\subset \AV$ is an $AV$-Dirac structure which is $G$ invariant, then it is clear that $L/G\subset \AV/G$ is an $A\!/\!G\;V\!/\!G$-Dirac structure (see Example~\ref{principal}).
\end{example}

\begin{example}[$\mathcal{E}^1(M)$-Dirac Structures]\label{e1mdirac}
Using Example~\ref{e1mcourant}, we can describe $\mathcal{E}^1(M)$, the bundle introduced by A. Wade in \cite{W00}, as an $AV$-Courant algebroid. In this context, the $\mathcal{E}^1(M)$-Dirac structures (also introduced by A. Wade in \cite{W00}) correspond directly to the $AV$-Dirac structures.
\end{example}

\section{Transitive Lie Algebroids}\label{princconv}
\subsection{Simplifying $AV$-Courant Algebroids}

Suppose that $A$ is a Lie algebroid, $V$ is an $A$-module, and $\AV$ is an $AV$-Courant algebroid over $M$ (where we use the notation given in the definition of $AV$-Courant algebroids). We will assume for the duration of this section that $M$ is connected, and we require that $A$ be a transitive Lie algebroid, namely the anchor map $a:A\to TM$ is surjective (see \cite{M05} for more details).

 Since $\AV$ may be quite complicated, we wish to examine whether this $AV$-Courant algebroid is the quotient of a much simpler $A'V'$-Courant algebroid on a principal bundle over $M$, where $A'$ is a very simple Lie algebroid and $V'$ is a very simple $A'$-module. To be more explicit, we wish to examine whether $\AV$ results from the construction in Example~\ref{TPWcourant}. For this to be true, it is clearly necessary that $A$ be the Atiyah algebroid of that principal bundle; namely if $P$ is the principal bundle, then $A=TP/G$. The existence of such a principal bundle is equivalent to the integrability of $A$ as a Lie algebroid:

\begin{proposition}
 Suppose that $A\to M$ is an integrable transitive Lie algebroid, that is to say, there exists a source-simply connected Lie groupoid $\xymatrix@1{\Gamma\ar@<.5ex>[r]^s\ar@<-.5ex>[r]_t &M}$ with Lie algebroid $A$ (see \cite{M05} for more details). Then $A$ is the Atiyah algebroid of a principal bundle.

 Conversely, if $A$ is the Atiyah algebroid of a principal bundle, then $A$ is an integrable Lie algebroid.
\end{proposition}

\begin{proof}
Suppose first that $A$ is integrable, then using the notation in the statement of the proposition, where $s:\Gamma\to M$ is the source map and $t:\Gamma\to M$ is the target map, let $x\in M$, let $P=\Gamma_x:=s^{-1}(x)$, and let $G=\Gamma_x^x:=s^{-1}(x)\cap t^{-1}(x)$.

Since $A$ is transitive, $t:P\to M$ is a surjective submersion, for clarity, we define $p:=t\rvert_P$. Furthermore, if $y\in M$, and $g\in\Gamma_x^y$, then $g:p^{-1}(x)\to p^{-1}(y)$ is a diffeomorphism, so $p:P\to M$ is a fibre bundle, with its fibre diffeomorphic to $G$. In addition, $G$ has a right action on $P$, given by right multiplication in the Lie groupoid. If $p^{-1}(y)=\Gamma_x^y$ is a fibre, and $g\in \Gamma_x^y$ then the diffeomorphism $g:p^{-1}(x)\to p^{-1}(y)$ is given by left groupoid multiplication while the action of $G$ on $P$ is given by right groupoid multiplication, so it is clear that the two operations commute, from which it follows that $G$ preserves the fibres of $P$, acting transitively and freely on them. Thus $P$ is a principal $G$ bundle.

Since $A$ is the Lie algebroid of $\Gamma$, it can be identified with the right invariant vector-fields on $\Gamma$ tangent to the source fibres. However, since $A$ is transitive, any two source fibres are diffeomorphic by right multiplication by some element. Thus $A$ can be identified with the $G$ invariant vector fields on $P$.

Conversely, if $A$ is the Atiyah algebroid of some principal bundle, it obviously integrates to the gauge groupoid associated to that principal bundle (see \cite{DW00} or Remark~\ref{integrabilityrmk}), and we may take $\Gamma$ to be the source-simply connected cover of the gauge groupoid.
\end{proof}

We now examine whether $V$ is an associated vector bundle.

\begin{proposition}
Suppose that $A$  is an integrable transitive Lie algebroid, and $V\to M$ is an $A$-module. Then there exists a (possibly disconnected) Lie group $G$, and a simply connected principal $G$-bundle $P\to M$ such that $V$ is the quotient bundle of $P\times \mathbb{R}^k$, for some $G$ action on $\mathbb{R}^k$. In this setting, the standard action of $\mathcal{X}(P)$ on $C^\infty(P,\mathbb{R}^k)$ induces the module structure on $V$.
\end{proposition}

\begin{proof}
Using the notation and the Lie groupoid described in the previous proposition, we consider $\Gamma_x\times V_x$, where $V_x$ is the fibre of $V$ at $x$. We may assume that $\Gamma$ is source simply connected and consequently since $V$ is an $A$-module, by Lie's second theorem there exists a Lie groupoid morphism $\Gamma\to\GL(V)$.\footnote{See, for instance, \cite{CF03}, \cite{MX00}, or \cite {MM02}, for more details. Here $\GL(V)$ is the Lie groupoid of linear isomorphisms of the fibres of $V$, namely $\GL(V)_x^y=\Hom(V_x,V_y)$.} Thus $\Gamma$ acts on $V$ and we have a map $\tilde p:\Gamma_x\times V_x\to V$ given by $(g,v)\to gv$, this is clearly a surjective submersion\footnote{Since $A$ is transitive and $M$ is connected, $t:\Gamma_x\to M$ is a surjective submersion. Let $y\in M$, and $\sigma:U\to \Gamma_x$ be a section (so that $t\circ\sigma=\id$). Then $(z,v)\to\sigma(z)(v):U\times V_x\to V_U$ is a diffeomorphism.}. Furthermore, 
$$\tilde p(g,v)=\tilde p(g',v')\Leftrightarrow g^{-1}g'\in\Gamma_x^x\text{ and }v=(g^{-1}g')v'.$$
 Thus, letting $G=\Gamma_x^x$ and $P=\Gamma_x$, we have $V\simeq (\Gamma_x\times V_x)/G\simeq (P\times V_x)/G$.

Furthermore, identifying $V_x$ with $\mathbb{R}^k$, if $X\in\mathcal{X}(P)\simeq\mathcal{X}(\Gamma_x)$, and $\sigma\in C^\infty(P,\mathbb{R}^k)$, then the standard action of $X$ on $\sigma$ is given by $X(\sigma)_z=\frac{\partial}{\partial t}\rvert_{t=0}\sigma(e^{tX}z)$ for any $z\in P\simeq\Gamma_x$. If we suppose that $X$ and $\sigma$ are $G$ invariant, then 
$$\tilde p(\frac{\partial}{\partial t}\rvert_{t=0}\sigma(e^{tX}(z)))=\frac{\partial}{\partial t}\rvert_{t=0}(e^{-tX}\tilde p(\sigma))_{p(z)}=(\Lied_X\tilde p(\sigma))_{p(z)},$$
since we defined the action of $\Gamma$ on $V$ in terms of the $A$-module structure of $V$.
\end{proof}

\begin{proposition}\label{liftprop}
Suppose that $A$ is an integrable Lie algebroid, and $V\to M$ is an $A$-module. Then $\AV$ results from the construction given in Example~\ref{TPWcourant}. Namely  there exists a Lie group $G$, and a principal $G$-bundle $P\to M$ such that $\AV$ is the quotient of a $T\!P\:\mathbb{R}^k$-Courant algebroid Furthermore, if $L\subset \AV$ is an $AV$-Dirac structure, then it is also the quotient of a corresponding  $T\!P\:\mathbb{R}^k$-Dirac structure on $P$.

Consequently, if $V$ is a line-bundle, then $\AV$ is simply the quotient of an exact Courant algebroid on $P$.
\end{proposition}

\begin{proof}
We choose some isotropic splitting of $\AV$, so that
$$\AV\simeq A\oplus(V\otimes A^*),$$
 the bracket is given by Equation~(\ref{bracket}), and the symmetric bilinear form by Equation~(\ref{pairing}). Then we can use the previous propositions to lift the right hand side to a principal bundle:

By the above propositions, there exists a (possibly disconnected) Lie group $G$, and a simply connected $G$-principal bundle, $\nu:P\to M$, such that $A\simeq TP/G$, and in addition to this there is a $G$-action on $W:=\mathbb{R}^{\dim(V)}$, say $\lambda:G\to \GL(W)$ such that $V=P\times_G W$. In this setting, $\Gamma(V\otimes\wedge^i A^*)\simeq\Omega^i(P,W)^G$, and $d:\Gamma(V\otimes\wedge^i A^*)\to\Gamma(V\otimes\wedge^{i+1}A^*)$ is the restriction of the exterior derivative $d$ to $\Omega^*(P,W)^G$.

Thus since $H\in\Gamma(V\otimes\wedge^3 A^*)\simeq\Omega^3(P,W)^G$, it is clear that we may view $H$ as a $G$-invariant element of $\Omega^3(P,W)$, and define the $T\!P\:W$-Courant algebroid $W\otimes T^*P\to\mathbb{T}\to TP$ in terms of it: Namely, $\mathbb{T}\simeq TP\oplus(W\otimes T^*P)$ endowed with a $W$-valued symmetric bilinear form given by Equation~(\ref{pairing}), and the bracket given by Equation~(\ref{bracket}). (See Example~\ref{TPWcourant} for more details on this construction.)

It is clear that $\AV$ is the quotient of this $T\!P\:W$-Courant algebroid.

Equivalently, it is easy to see that $TP=\nu^*A$, $W=\nu^*V$ and $\mathbb{T}=\nu^*\AV$. The $W$-valued symmetric bilinear form on $\mathbb{T}$ is simply the pullback of the $V$-valued symmetric bilinear form on $\AV$, and if $e_1,e_2\in\Gamma(\AV)$, then $\lb \nu^*e_1,\nu^*e_2 \rb=\nu^*\lb e_1,e_2 \rb$, and the bracket on $\mathbb{T}$ is then extended to arbitrary sections of $\mathbb{T}$ by Axioms~(AV-\ref{AX6})~and Remark~\ref{brktDer}.

Next, let $\tilde L=\nu^*(L)\subset\mathbb{T}$. It is obvious that $L^\perp=L\Rightarrow\tilde L^\perp=\tilde L$; and similarly since $L$ is involutive, so is $\tilde L$.

Thus $\tilde L\subset \mathbb{T}$ is a $T\!P\:W$-Dirac structure, and ${\tilde L}/G=L$.
\end{proof}

\begin{example}
If $A=TM$ and $V$ is a flat vector bundle over $M$, then following the proof of Proposition~\ref{liftprop} we see that $G=\pi_1(M)$ is the fundamental group, and $P=\tilde M$ is the simply connected covering space of $M$ over which the pullback of $V$ is a trivial vector bundle.
\end{example}

\begin{remark}\label{integrabilityrmk}
The above propositions construct the principal bundle $P$, and the Lie group $G$. Suppose however, that we already have a Lie group $G'$, and a connected $G'$-principal bundle $\nu':P'\to M$ such that $A\simeq TP'/G'$. It will not be difficult to see that $\AV$ is the quotient of a $AV$-Courant algebroid on $P'$.

Let $\mathcal{G}=(P'\times P')/G'$, where we take the quotient by the diagonal action. Then 
$$\xymatrix{\mathcal{G}\ar@<.5ex>[r]^{s}\ar@<-.5ex>[r]_{t} &M}$$
 is a Lie groupoid with Lie algebroid $A$, where the source map is $s:[u,v]\to \nu'(v)$, the target map is $t:[u,v]\to \nu'(u)$, and the multiplication is $[u,v]\cdot[v,w]=[u,w]$.\footnote{An element of $\mathcal{G}$ is an equivalence class, which we may view as a subset of $\nu'^{-1}(y)\times \nu'^{-1}(z)$ which is $G$ invariant. As such, we may view it as the graph of an equivariant diffeomorphism $\nu'^{-1}(y)\to \nu'^{-1}(z)$. The multiplication in $\mathcal{G}$ is simply the composition of these diffeomorphisms. See \cite{DW00} for details} Hence by Lie's second theorem, (see \cite{CF03}, \cite{MX00},  or \cite{MM02}, for instance for more details) since $\Gamma$, the Lie groupoid used in the proof of Proposition~\ref{liftprop}, is source simply connected, there is a unique Lie groupoid morphism $\Phi:\Gamma\to \mathcal{G}$ which restricts to the identity map on the Lie algebroid $A$.

It follows that $\Phi\rvert_{P}:P\to P'$ is a covering map,\footnote{Here we use the identifications $P=\Gamma_x$ and $P'=\mathcal{G}_x$. It is a covering map since the right invariant vector fields, which are identified with the sections of $A$, span the tangent space of the source fibres.} and $\Phi\rvert_{G}:G\to G'$ is a covering morphism of Lie groups.\footnote{Here we use the identifications $G=\Gamma_x^x$ and $G'=\mathcal{G}_x^x$.} It is easy to see that $H=\ker(\Phi\rvert_{G})\simeq \pi(P')$, and $P'=P/H$.

Thus, we may take the quotient of the $T\!P\:W$-Courant algebroid on $P$ (constructed in Proposition~\ref{liftprop}) by $H$, to form a $T\!P'\;W\!/\!H$-Courant algebroid on $P'$ whose quotient by $G'$ is $\AV$. It is important to note that while $W$ is a trivial vector bundle, $W/H$ is a flat vector bundle.
\end{remark}

\begin{corollary}\label{localint}
Suppose that $V$ is an $A$-module, and $M$ is contractible, then $\AV$ is the quotient of a $T\!P\:\mathbb{R}^k$-Courant algebroid $\mathbb{R}^k\otimes T^*P\to\mathbb{T}\to TP$ on some principal $G$-bundle, $P$. (See Example~\ref{principal}). Furthermore, if $L\subset \AV$ is an $AV$-Dirac structure, then it is also the quotient of a $T\!P\:\mathbb{R}^k$-Dirac structure $\tilde L\subset \mathbb{T}$.
\end{corollary}

\begin{proof}Every transitive Lie algebroid is integrable over a contractible space, see \cite{M05} for details.\end{proof}

\subsection{Contact Manifolds}

Iglesias and Wade show how to describe contact manifolds as $\mathcal{E}^1(M)$-Dirac Structures in \cite{IW05}. Thus in light of Example~\ref{e1mdirac}, we can describe them as $AV$-Dirac structures. We will now describe this same construction from a more geometric perspective, similar to their description in \cite{IW06}:

To simplify things, we assume that $(M,\xi)$ is a co-oriented contact manifold, and we use the fact that there is a one-to-one correspondence between co-oriented contact manifolds and symplectic cones; namely, $(N,\omega_N)$ is a symplectic cone, where $N=\Ann^+(\xi)\subset T^*M$, $\omega_N$ is the two form induced from the standard symplectic form on $T^*M$ and the action is scalar multiplication along the fibres.

Now $X\to\iota_X\omega_N:TN\to T^*N$ defines an isomorphism. We let $L\subset TN\oplus T^*N$ be the graph of this morphism; and it is easy to check that $L$ is a maximal isotropic subbundle of $TN\oplus T^*N$, and since $\omega_N$ is closed, $L$ is a Dirac subbundle of the standard Courant algebroid on $N$. To summarize, $(M,\xi)$ is a contact manifold if and only if $L$ is the graph of an isomorphism, or simply $L\cap TN=0$.

We may make the identifications $N\simeq M\times\mathbb{R}^+$ and $x,t\to x,\ln(t): M\times\mathbb{R}^+\simeq M\times\mathbb{R}$; consequently, as described in Example~\ref{e1mprincipal}, the quotient of the standard Courant algebroid on $N=M\times\mathbb{R}$ by the $\mathbb{R}$ action (\ref{e1maction}) yields an $\mathcal{E}^1(M)$ bundle on $M$; or an $AV$-Courant algebroid where $A=TN/\mathbb{R}$ and $V$ is the trivial line bundle on $M$.

Since $\omega_N$ is preserved by this action, it follows that its graph, $L$, is $\mathbb{R}$-invariant and defines an $\mathcal{E}^1(M)$-Dirac structure which we denote by $\tilde{L}_\xi$. It is perhaps important to note that $\tilde{L}_\xi$ is defined intrinsically. We may conclude that:

\begin{proposition}
 $(M,\xi)$ is a contact manifold if and only if $\tilde{L}_\xi\cap A=0$ (under the canonical splitting).
\end{proposition}

\section{CR-structures and Courant Algebroids}\label{CRstructures}

Suppose $M$ is a smooth manifold, let $H\subset TM$ be a subbundle and suppose $J\in\Gamma(\Hom(H,H))$ is such that $J^2=-\id$. Then $(H,J)$ is called an almost CR structure. We let $H_{1,0}\subset \mathbb{C}\otimes H\subset \mathbb{C}\otimes TM$ denote the $+i$-eigenbundle of $J$, if $H_{1,0}$ is involutive, then it is called a CR-structure. It is possible to describe this as a Courant algebroid:

 We consider the bundle $H^*\oplus H\simeq  T^*M\oplus H/\Ann(H)$, and the bundle map $\mathbb{J}:=-J^*\oplus J\in\Gamma(\Hom(H^*\oplus H,H^*\oplus H^*))$. It is clear that $\mathbb{J}^2=-\id$. Let $L=\ker(\mathbb{J}-i)\oplus\Ann(H)\subset \mathbb{C}\otimes(TM\oplus T^*M)$.

\begin{proposition}
$L$ is involutive under the standard Courant bracket if and only if $J$ defines a CR structure.
\end{proposition}

\begin{proof}
We notice that $L=H_{1,0}\oplus\Ann(H_{1,0})$. Therefore $L$ is involutive under the Courant bracket only if $\pi(L)=H_{1,0}$ is involutive, where $\pi:TM\oplus T^*M\to TM$ is the projection. Thus $J$ defines a CR structure.

Conversely, suppose that $H_{1,0}$ is involutive. Then if $I$ is the ideal generated by $\Ann(H_{1,0})$ in $\Gamma(\mathbb{C}\otimes\wedge T^*M)$, then $I$ is closed under the differential: $dI\subset I$.

In particular, if we restrict our attention to a local neighborhood on $M$, and  $\alpha_i$ is a local basis for $\Ann(H_{1,0})$ and $\xi\in \Gamma(\Ann(H_{1,0}))$, then $d\xi=\sum_i \beta_i\wedge\alpha_i$ for some $\beta_i\in\Omega^1(M,\mathbb{C})$. Thus, for any $X\in \Gamma(H_{1,0})$, we have,
$$\iota_Xd\xi=\sum_i \beta_i(X)\alpha_i\in \Gamma(\Ann(H_{1,0})),$$
and
$$\Lied_X\xi=d\iota_X\xi+\iota_Xd\xi=\iota_Xd\xi\in \Gamma(\Ann(H_{1,0})).$$
It follows that $L$ is involutive under the standard Courant bracket.

\end{proof}

In the next section we shall generalize this construction.

\section{Generalized CR structures}

Suppose that $M$ is a manifold, $A$ is a Lie algebroid over $M$, $V$ an $A$-module of rank one over $M$, and $\AV$ an $AV$-Courant algebroid over $M$. Suppose further that $A$ has some distinguished subbundle $H\subset A$, and consider the bundle given by
$$\mathbb{H}=q(\pi^{-1}(H)), \text{ where } q:\pi^{-1}(H)\to \pi^{-1}(H)/\jpi(V\otimes\Ann(H)).$$
Then the pairing on $\AV$ restricts non-degenerately to $\mathbb{H}$, and we have an exact sequence 
$$0\to V\otimes H^*\xrightarrow{\jpi}\mathbb{H}\xrightarrow{\pi}H\to 0.$$
\begin{definition}
$\mathbb{J}\in\Gamma(\Hom(\mathbb{H},\mathbb{H}))$ is called a generalized CR structure if:

\begin{enumerate}
\item $\mathbb{J}$ is orthogonal (preserves the pairing on $\mathbb{H}$)
\item $\mathbb{J}^2=-1$
\item $L:=q^{-1}(\ker(\mathbb{J}-i))\subset \mathbb{C}\otimes \AV$ is involutive.
\end{enumerate}
\end{definition}

\begin{remark}
$L:=q^{-1}(\ker(\mathbb{J}-i))\subset \mathbb{C}\otimes \AV$ is a maximal isotropic subspace of $\AV$ since $\ker(\mathbb{J}-i)$ is a maximal isotropic subspace of $\mathbb{H}$. In particular, since we assume that $L$ is involutive, it is an $AV$-Dirac structure.
\end{remark}

\begin{remark}
Here we have relaxed the requirement $L\cap \bar L=0$ in the definition of a generalized complex structure. While we have allowed $L\cap \bar L$ to be non-trivial, it must lie in $j(V\otimes \Ann(H))\subset V\otimes A^*$. As pointed out in Remark~\ref{brktDer}, this can be interpreted as saying that $L\cap \bar L$ only fails to be trivial up to an `infinitesimal'. On the other hand, we still require $L$ to be an $AV$-Dirac structure.

This is in contrast to the approach taken by generalized CRF structures, introduced by Izu Vaisman in \cite{V08}, which requires $L\cap \bar L=0$, but does not require $L$ to be a Dirac structure.
\end{remark}

 It is well known that one can canonically associate a Poisson structure to every generalized complex structure. The analogue for generalized CR structures is to endow $V\otimes A^*$ with a non-trivial Lie algebroid structure, which we shall do in a canonical fashion following the corresponding argument given for generalized complex structures in \cite{G07}.

We have an inclusion $i:H\to A$, and consequently a map $\mathbb{J}\circ \jpi\circ (\id\otimes i^*):V\otimes A^*\to \mathbb{H}$, which (abusing notation), we shall simply call $\mathbb{J}$. We consider the family of subspaces of $\AV$ given by 
$$D_t:=e^{t\mathbb{J}}(V\otimes A^*)+ V\otimes\Ann(H)=q^{-1}(e^{t\mathbb{J}}(V\otimes H^*)).$$
 Since $e^{t\mathbb{J}}=\cos(t)+\sin(t)\mathbb{J}:\mathbb{H}\to \mathbb{H}$ is orthogonal, and $\jpi(V\otimes H^*)$ is a lagrangian subspace of $\mathbb{H}$, it follows that $D_t$ is lagrangian for each $t$.

The following proposition is a slight generalization of a result of Gualtieri \cite{G07}.

\begin{proposition}(Gualtieri)
The family $D_t$ of almost $AV$-Dirac structures is integrable for all $t$.
\end{proposition}

\begin{proof}

Let $\xi_1,\xi_2\in\Gamma(V\otimes A^*)$, then since $V\otimes A^*\subset L\oplus \bar L$, we may choose $X_j\in \Gamma(L)$, and $Y_j\in \Gamma(\bar L)$, such that $\xi_j=X_j+Y_j$. It follows that $\mathbb{J}\xi_j=iX_j-iY_j+V\otimes\Ann(H)$. In fact, since $L\cap\bar L=V\otimes\Ann(H)$, by choosing $X_j$ and $Y_j$ appropriately, we may suppose that $iX_j-iY_j$ is any given representative of $\mathbb{J}\circ i^*(\xi_j)$ in $\pi^{-1}(H)$. Abusing notation will use the term $\mathbb{J}(\xi_j)$ and our particular choice of representative $iX_j-iY_j$ interchangeably. Then,
$$\begin{array}{l}
\lb \mathbb{J}\xi_1, \mathbb{J}\xi_2\rb-\lb\xi_1,\xi_2\rb \\
=\lb iX_1-iY_1,iX_2-iY_2\rb-\lb X_1+Y_1,X_2+Y_2\rb\\
=-2\lb X_1,X_2\rb-2\lb Y_1,Y_2\rb\\
\end{array}$$
and
$$\begin{array}{l}
\lb \mathbb{J}\xi_1, \xi_2\rb-\lb\xi_1,\mathbb{J}\xi_2\rb \\
=\lb iX_1-iY_1,X_2+Y_2\rb-\lb X_1+Y_1,iX_2-iY_2\rb\\
=2i\lb X_1,X_2\rb-2i\lb Y_1,Y_2\rb\\
\end{array}$$
Thus, since $L$ and hence $\bar L$ are involutive, we have $\lb \mathbb{J}\xi_1, \mathbb{J}\xi_2\rb-\lb\xi_1,\xi_2\rb +V\otimes\Ann(H)= \mathbb{J}(\lb \mathbb{J}\xi_1, \xi_2\rb-\lb\xi_1,\mathbb{J}\xi_2\rb)+V\otimes\Ann(H)$

We let $a=\cos(t)$ and $b=\sin(t)$, and we have,
$$\begin{array}{l}
\lb(a+b\mathbb{J})\xi_1,(a+b\mathbb{J})\xi_2\rb\\
=ab(\lb \xi_1,\mathbb{J}\xi_2\rb+\lb\mathbb{J}\xi_1,\xi_2\rb)+b^2\lb \mathbb{J}\xi_1, \mathbb{J}\xi_2\rb\\
=ab(\lb \xi_1,\mathbb{J}\xi_2\rb+\lb\mathbb{J}\xi_1,\xi_2\rb)+b^2(\lb \mathbb{J}\xi_1, \mathbb{J}\xi_2\rb-\lb\xi_1,\xi_2\rb)\\
\end{array}$$
 So modulo $V\otimes\Ann(H)$, we see that
$$\begin{array}{l}
\lb(a+b\mathbb{J})\xi_1,(a+b\mathbb{J})\xi_2\rb +V\otimes\Ann(H)\\
=b(a+b\mathbb{J})(\lb \xi_1,\mathbb{J}\xi_2\rb+\lb\mathbb{J}\xi_1,\xi_2\rb)+V\otimes\Ann(H)\\
\end{array}$$
Since $\lb \xi_1,\mathbb{J}\xi_2\rb+\lb\mathbb{J}\xi_1,\xi_2\rb\in V\otimes A^*$ it follows that $(\cos(t)+\sin(t)\mathbb{J})(V\otimes A^*)+ V\otimes\Ann(H)$ is involutive. 

\end{proof}

We next consider the map $P:V\otimes A^*\to H\xrightarrow{i} A$, which for $\xi,\eta\in V\otimes A^*$, is given by
$$\langle P(\xi),\eta\rangle=\langle\frac{\partial}{\partial t}\rvert_{t=0} e^{t\mathbb{J}}(\xi),\eta\rangle =\langle\mathbb{J}\xi,\eta\rangle \qquad(=\langle i\circ\pi\circ\mathbb{J}\circ \jpi\circ i^*(\xi),\eta\rangle ).$$
Clearly, since $\mathbb{J}$ is an orthogonal almost complex structure on $\mathbb{H}$, $P$ will be given by an element of $\Gamma(V^*\otimes\wedge^2 A)$, which we will also denote by $P$. Adapting a proposition given in \cite{G07}, we have:

\begin{proposition}(Gualtieri)\label{poisson}
The bivector field $P=i\circ\pi\circ\mathbb{J}\circ \jpi\circ i^*: V\otimes A^*\to A$ defines a Lie algebroid structure on $V\otimes A^*$: The bracket is given by 
$$[\xi,\eta]=\iota_{P(\cdot,\xi)}d\eta-\iota_{P(\cdot,\eta)}d\xi+d(P(\xi,\eta))),$$
where $\xi,\eta\in V\otimes A^*$, and the anchor map by $\xi\to a\circ P(\xi,\cdot)V\otimes A^*\to TM$, where $a:A\to TM$ is the anchor map of $A$. Furthermore, the map $\xi\to a\circ P(\xi,\cdot):V\otimes A^*\to A$ is a Lie-algebroid morphism.

\end{proposition}

The proof is an adaptation of one found in \cite{G07}.

\begin{proof}
We choose a splitting of the $AV$-Courant algebroid, and use the isomorphism and notation described in Proposition~\ref{mainprop}. Then if we choose $t$ sufficiently small, the $AV$-Dirac structures $D_t$ can be described as the graphs of $\beta_t\in\Gamma(V^*\otimes\wedge^2 A)$.

In \cite{SW01}, it is shown that the integrability condition of a twisted Poisson structure $\beta$ over a 3-form background $\gamma$, is $[\beta,\beta]=\wedge^3\tilde\beta (\gamma)$, where $\tilde\beta:T^*M\to TM$ is given by $\tilde\beta(\xi)(\eta)=\beta(\xi,\eta)$. We would like to derive a similar equation for $\beta_t$, but we have not defined a bracket for sections of $V^*\otimes\wedge^2 A$. In order to define such a bracket, we first define a sheaf of rings over $M$:

We let $\mathcal{F}:=(S(V)\otimes S(V^*))/I$, where $S(V)$ denotes the symmetric algebra generated by $V$, and $I$ is the ideal generated by $u\otimes f - f(u)$ for $f\in \Gamma(V^*)$ and $u\in \Gamma(V)$. Since $V$ is one dimensional, if $t\in\Gamma(V)$ is a local basis then $\mathcal{F}$ is locally isomorphic to $C^\infty(M)[t,t^{-1}]$ as a ring. It is clear that it has a well defined $\mathbb{Z}$ grading, which for a homogeneous $v\in\mathcal{F}$, we denote by $\tilde v$

$\Gamma(S(V)\otimes S(V^*))$ is a $\Gamma(A)$ module, where sections of $\Gamma(A)$ act as derivations, and it is easy to check that $\Gamma(I)$ is a sub-module. Thus it is clear that $\Gamma(A)$ acts on $\Gamma(\mathcal{F})$ by derivations satisfying the Leibniz rule with respect to the ring structure on $\mathcal{F}$.

We define a bracket on $\mathcal{F}\otimes\wedge^* A$, as follows (for $v,w\in \Gamma(\mathcal{F})$ and $P,Q\in\Gamma(\wedge^*A)$):

\begin{itemize}
\item $[X,v]=Xv$ for any $X\in \Gamma(A)$, and $[v,w]=0$
\item $[P\wedge Q,v]=P\wedge[Q,v]+(-1)^{|Q|}[P,v]\wedge Q$
\item $[P,Q]$ is given by the Schouten-Nijenhuis bracket.
\item $[vP,wQ]=(v[P,w])Q-(-1)^{(|P|-1)(|Q|-1)}(w[Q,v])P+vw[P,Q]$.
\end{itemize}

If we write $\abs{vP}=i$ for $P\in\wedge^i A$, and $\deg(vP)=(\tilde v,\abs{vP})$, then it is clear that our bracket satisfies the following identities (for homogeneous $a,b,c\in \Gamma(\mathcal{F}\otimes\wedge^* A)$):

\begin{itemize}
\item $\deg(ab)=\deg(a)+\deg(b)$ and $\deg([a,b])=\deg(a)+\deg(b)-(0,1)$
\item $(ab)c=a(bc)$ and $ab=(-1)^{\abs{a}\abs{b}}ba$
\item $[a,bc]=[a,b]c+(-1)^{(\abs{a}-1)\abs{b}}b[a,c]$
\item $[a,b]=-(-1)^{(\abs{a}-1)(\abs{b}-1)}[b,a]$
\item $[a,[b,c]]=[[a,b],c]+(-1)^{(\abs{a}-1)(\abs{b}-1)}[b,[a,c]]$
\end{itemize}

We next extend $d$ to a map $d:\mathcal{F}\otimes\wedge^i A^*\to\mathcal{F}\otimes\wedge^{i+1}A^*$ in the obvious way. We also have a natural $\mathcal{F}$-bilinear pairing on $\Gamma(\mathcal{F}\otimes\wedge^* A^*)\times\Gamma(\mathcal{F}\otimes\wedge^* A)$, which for $v_i,w_j\in\mathcal{F}$, $\alpha_i\in\Gamma(A^*)$, and $X_j\in\Gamma(A)$, is given by
$$\langle(v_1\otimes\alpha_1)\cdots(v_p\otimes\alpha_p),(w_1\otimes X_1)\cdots(w_q\otimes X_q)\rangle ={\Bigg\{\begin{array}{ll}
0&\text{if }p\neq q\\
\det(v_iw_j\otimes\alpha_i(X_j))&\text{if }p=q\end{array}}.$$
We define a morphism $\iota:\mathcal{F}\otimes\wedge^* A\to\End(\mathcal{F}\otimes\wedge^* A^*)$ by $\langle\xi,PQ\rangle =\langle\iota_P\xi,Q\rangle $. For $P\in\mathcal{F}\otimes A$, $\iota_P$ is a derivation.

We also define a morphism $\breve{\iota}:\mathcal{F}\otimes\wedge^* A^*\to\End(\mathcal{F}\otimes\wedge^* A)$ by $\langle\xi\eta,P\rangle =\langle\xi,\breve{\iota}(\eta)P\rangle $. For $\alpha\in\mathcal{F}\otimes A^*$, $\breve{\iota}(\alpha)$ is a derivation on the right. Namely, $\breve{\iota}(\alpha)(PQ)=P \breve{\iota}(\alpha)Q+(-1)^{\abs{Q}}(\breve{\iota}(\alpha)P)Q$ (where $P,Q\in \mathcal{F}\otimes\wedge^* A$ are homogeneous).

Next, we notice that $\iota_{[P,Q]}=-[[\iota_Q,d],\iota_P]$. This is easy to check, following the argument given in \cite{M97}. Also following an argument in \cite{M97} one can verify that, for $\eta\in\Gamma(\mathcal{F}\otimes A^*)$,
$$
\breve{\iota}(\eta)[P,Q]-[P,\breve{\iota}(\eta)Q]-(-1)^{\abs{Q}-1}[\breve{\iota}(\eta)P,Q]
=(-1)^{\abs{Q}-2}(\breve{\iota}(d\eta)(PQ)-P \breve{\iota}(d\eta)Q-(\breve{\iota}(d\eta)P)Q).
$$
From this, we calculate, for any $\beta\in\Gamma(\mathcal{F}\otimes\wedge^2 A)$ and $\xi,\eta\in\Gamma(\mathcal{F}\otimes A^*)$,
$$
[\breve{\iota}(\xi)\beta,\breve{\iota}(\eta)\beta]=\frac{1}{2}\breve{\iota}(\xi\eta)[\beta,\beta]+[\beta,\langle\eta\xi,\beta\rangle ]
+\frac{1}{2}(\breve{\iota}(\eta d\xi)\beta^2-\breve{\iota}(\xi d\eta)\beta^2)-\langle d\xi,\beta\rangle \breve{\iota}(\eta)\beta+\langle d\eta,\beta\rangle \breve{\iota}(\xi)\beta
.$$
Furthermore, it is not difficult to verify that $[\beta,\langle\eta\xi,\beta\rangle ]=\breve{\iota}(d\beta(\eta,\xi))\beta$, while $$\frac{1}{2}(\breve{\iota}(\eta d\xi)\beta^2-\breve{\iota}(\xi d\eta)\beta^2)-\langle d\xi,\beta\rangle \breve{\iota}(\eta)\beta+\langle d\eta,\beta\rangle \breve{\iota}(\xi)\beta=\breve{\iota}(\iota_{\breve{\iota}(\xi)\beta}d\eta-\iota_{\breve{\iota}(\eta)\beta}d\xi)\beta.$$ 

Thus, we have, for $\beta\in \Gamma(V^*\otimes\wedge^2 A)$,
$$\begin{array}{rl}
&\lb -\breve{\iota}(\xi)\beta+\xi,-\breve{\iota}(\eta)\beta+\eta\rbp\\
=&[\breve{\iota}(\xi)\beta,\breve{\iota}(\eta)\beta]-\iota_{\breve{\iota}(\xi)\beta}d\eta+\iota_{\breve{\iota}(\eta)\beta}d\xi+d(\beta(\xi,\eta))+\iota_{\breve{\iota}(\xi)\beta}\iota_{\breve{\iota}(\eta)\beta}H\\
=&\breve{\iota}(\iota_{\breve{\iota}(\xi)\beta}d\eta-\iota_{\breve{\iota}(\eta)\beta}d\xi-d(\beta(\xi,\eta)))\beta-\iota_{\breve{\iota}(\xi)\beta}d\eta+\iota_{\breve{\iota}(\eta)\beta}d\xi+d(\beta(\xi,\eta))\\
+&\frac{1}{2}\breve{\iota}(\xi\eta)[\beta,\beta]+\iota_{\breve{\iota}(\xi)\beta}\iota_{\breve{\iota}(\eta)\beta}H.
\end{array}$$
It follows that $\beta_t$ defines an $AV$-Dirac structure under our chosen splitting if and only if $\frac{1}{2}\breve{\iota}(\eta\xi)[\beta_t,\beta_t]=\breve{\iota}(\iota_{\breve{\iota}(\xi)\beta_t}\iota_{\breve{\iota}(\eta)\beta_t}H)\beta_t$. To rewrite this, we let $\tilde\beta:\mathcal{F}\otimes A^*\to\mathcal{F}\otimes A$ be the map $\alpha\to -\breve{\iota}(\alpha)\beta$. The condition is then
$$[\beta_t,\beta_t]=2\wedge^3 \tilde\beta_t (H).$$
We differentiate both sides by $t$, and evaluate at 0. Since we have $P=\frac{\partial}{\partial t}\big\rvert_0 \beta_t$ and $\beta_0=0$, the cubic term vanishes, and we see that the condition is
$$[P,P]=0.$$
The result follows immediately from this.

\end{proof}

We also have a bracket $\{,\}$ on $\Gamma(V)$, which for $v,w\in\Gamma(V)$ is given by

$\{v,w\}=P(dv,dw).$

It satisfies the following properties (for $f\in C^\infty(M)$):

\begin{itemize}
\item $\{\cdot,\cdot\}$ is bilinear.
\item $\{v,w\}=-\{w,v\}$
\item $\{v,fw\}=f\{v,w\}+(a\circ P(dv)(f))w$
\item $\{u,\{v,w\}\}=\{\{u,v\},w\}+\{v,\{u,w\}\}$ (for any $u,v,w\in\Gamma(V)$)
\end{itemize}

Since $V$ is a line-bundle, this is quite similar to a Poisson structure. In particular, if $U\subset M$ is an open set on which $\sigma\in\Gamma(V\rvert_U)$ is a local basis such that $P(\sigma)=0$, then we have a morphism
$$\rho:C^\infty(U)\xrightarrow{f\to f\sigma} \Gamma(V\rvert_U),$$
which allows us to define a Poisson structure on $U$, by
$$\{f,g\}=\rho^{-1}\{\rho(f),\rho(g)\}.$$
In particular, if in some neighborhood $U\subset M$, $V$ admits a non-zero $A$-parallel section $\sigma\in\Gamma(V\rvert_U)$, then $P(\sigma)=0$, and thus $U$ is endowed with a Poisson structure. In fact, the Poisson structure associated to $U$ is this way is unique up to a constant multiple. Furthermore, if it exists at one point on a leaf of $A$, then it exists for any neighborhood of any point in that leaf:

\begin{remark}[Poisson Structure on a Leaf of $A$]
 Suppose that $F\subset M$ is a connected leaf of the foliation given by $A$, then $a:A\rvert_F\to TF$ is a Lie algebroid, and we have an exact sequence of Lie algebroids given by $0\to L=\ker(a)\to A\rvert_F\to TF\to 0$, where $L$ is actually a bundle of Lie algebras. The following are equivalent:

\begin{itemize}
\item $V$ admits an $A\rvert_F$-parallel section for any neighborhood $U\subset F$.
\item $L$ acts trivially on $V\rvert_F$.
\item $L_x$ acts trivially on $V_x$, for some point $x\in F$.\footnote{This follows from the fact that for any $x,y\in F$ there is a Lie algebroid morphism of $A$ covering a diffeomorphism of $M$ which takes $x$ to $y$. In addition these morphisms can be assumed to come from flowing along a section of $A$, and hence extend to $V$.}
\end{itemize}

Note that, up to a constant multiple, there is a unique $A$-parallel section of $V\rvert_F$. Thus, if $\sigma\in\Gamma(V\rvert_F)$ is a non-zero $A$-parallel section we can associate a Poisson structure to $F$, unique up to a constant multiple.
\end{remark}

\begin{remark}[Jacobi Structure]
If $V$ does not admit $A$-parallel sections, but for some $U\subset M$ there is a canonical choice of a local basis $\sigma\in\Gamma(V\rvert_U)$, we may still consider the isomorphism
$$\rho:C^\infty(U)\xrightarrow{f\to f\sigma} \Gamma(V\rvert_U),$$
which allows us to define a bracket on $C^\infty(U)$ by
$$[f,g]_\sigma=\rho^{-1}\{\rho(f),\rho(g)\}.$$
One notices that this bracket endows $C^\infty(U)$ with a Lie algebra structure which is local in the sense that the linear operator
$$D_f:C^\infty(U)\xrightarrow{g\to[f,g]_\sigma} C^\infty(U)$$
is local for all $f\in C^\infty(U)$. It is an important result (See \cite{S74},\cite{K76} or \cite{G84}) that for any local Lie algebra structure, there exists unique $\Lambda\in\Gamma(\wedge^2 TM)$, and $E\in\Gamma(TM)$ with $[\Lambda,\Lambda]=-2\Lambda\wedge E$ and $[\Lambda,E]=0$ such that
$$[f,g]_\sigma=\{f,g\}_\Lambda+f\Lied_X g-g\Lied_X f,$$
where $\{f,g\}_\Lambda=\breve{\iota}_{df}\breve{\iota}_{dg}\Lambda$.

The triple $(U,\Lambda,E)$ is then called a Jacobi manifold. Note however the dependence of $\Lambda$ and $E$ on $\sigma$; this is unlike the local Poisson structure which (if it exists) is unique up to a constant multiple.
\end{remark}

\begin{example}[CR Structures]
As described in Section~\ref{CRstructures}, a CR-structure on a manifold $M$ can be described by a generalized CR structure. In this case, $V$ can be taken to be the trivial bundle, and $A$ can be taken to be $TM$. It follows from the above discussion that there is a Poisson structure $P\in\Gamma(\wedge^2 TM)$ associated to the CR structure.

If $L\subset \mathbb{C}\otimes TM$ is the CR-structure, and $H=\mathbf{Re}(L\oplus\bar L)\subset TM$, then $P(T^*M)\subset H$. So the symplectic foliation associated to $P$ is everywhere tangent to $H$.
\end{example}

\begin{example}[Quotients of Generalized Complex Structures]
If the procedures described in Example~\ref{principalDirac} and Example~\ref{principal} are applied to a generalized complex structure, then one obtains a generalized CR structure.
\end{example}

\begin{example}[Contact Structures and Generalized Contact Structures]
Suppose that $M$ is a contact manifold, then there is a canonical way to associate a generalized CR structure to $M$. In particular, if $N=M\times \mathbb{R}$ is its symplectization, then $N$ admits a generalized complex structure corresponding to its symplectic structure. $\mathbb{R}$ acts on $N$, and the quotient is a generalized CR structure on $M$ (In the sense of Example~\ref{e1mprincipal} and Example~\ref{e1mdirac})

This procedure is also described in \cite{IW05} and \cite{IW06}, where they describe it as a generalized contact structure. In fact any generalized contact structure results from the quotient of generalized complex structure, and as such can also be described as a generalized CR structure.

Since the Lie algebroid $A$ and the vector bundle $V$ describe an $\mathcal{E}^1(M)$ structure, as given in Example~\ref{e1mcourant}, it can be checked that $V$ does not admit parallel sections, and thus, in general, $P\in\Gamma(V^*\otimes\wedge^2 A)$ does not describe a Poisson structure, but rather a Jacobi structure. When the generalized contact structure is simply a contact structure, then $P$ corresponds to a Jacobi structure describing the contact structure.

To be more explicit, we let $M$ be a contact manifold with contact distribution $\xi\subset TM$, and $N=M\times\mathbb{R}$ its symplectization, where we let $t:M\times\mathbb{R}\to\mathbb{R}$ be the projection to the second factor, and $\omega\in\Omega^2(N)$ denote the corresponding symplectic form. (That is, $\omega = e^t(d\eta+dt\wedge \eta)$, where $\eta\in\Ann(\xi)$ is nowhere vanishing.) We note that $\Lied_{\frac{\partial}{\partial t}}\omega=\omega$.

Since $N$ is a symplectic manifold, we can associate to it a canonical generalized complex structure $\mathcal{J}:TN\oplus T^*N\to TN\oplus T^*N$ on the standard Courant algebroid 
$$0\to T^*N\to TN\oplus T^*N\to TN\to 0$$
(see \cite{G07} for details). 

The Poisson bivector $\pi\in\Gamma(\wedge^2 TN)$ associated to this generalized complex structure has the property that $\Lied_{\frac{\partial}{\partial t}}\pi=-\pi$ (since it is the Poisson bivector corresponding to $\omega$). It follows that we can write $\pi=e^{-t}(\Lambda+\frac{\partial}{\partial t}\wedge E)$ for $E\in\Gamma(M)$, and $\Lambda\in\Gamma(\wedge^2 M)$. Then $[\pi,\pi]=0$ implies that
$$
0=[\pi,\pi]=[e^{-t}(\Lambda+\frac{\partial}{\partial t}\wedge E),e^{-t}(\Lambda+\frac{\partial}{\partial t}\wedge E)]
=e^{-2t}[\Lambda,\Lambda]-2e^{-2t}\Lambda\wedge E+2e^{-2t}\frac{\partial}{\partial t}\wedge[\Lambda,E].
$$
From this it follows that $[\Lambda,\Lambda]=-2\Lambda\wedge E$ and $[\Lambda,E]=0$, which are the defining conditions for a Jacobi structure $(\Lambda,E)$ on $M$.

Now, we consider the $T\!M\oplus\mathbb{R}-\mathbb{R}$ Courant algebroid structure on $M$, given by taking the quotient by the $G=\mathbb{R}$ action on $N=M\times \mathbb{R}$,
$$0\to T^*N/G\to (TN\oplus T^*N)/G\to TN/G\to 0,$$
and the generalized CR structure on $M$ given by quotient homomorphism 
$$\mathbb{J}:=\mathcal{J}/G:(TN\oplus T^*N)/G\to (TN\oplus T^*N)/G.$$
 They define an $AV$-Courant algebroid, where $A=TN/G$, and the bundle $V\to M$ is trivial, with $\Gamma(V)\simeq C^\infty(N)^G$ (this is in fact an $\mathcal{E}^1(M)$ structure, see \cite{IW05}). Abusing notation, we denote by $e^t\in \Gamma(V)$ the section associated to the $G$-invariant function $e^t\in C^\infty(N)$.

Then the bivector $P\in\Gamma(V^*\otimes\wedge^2 A)$ associated to the generalized CR structure on $M$ is simply $e^{-t}(\Lambda+\frac{\partial}{\partial t}\wedge E)$, and it defines a Jacobi structure on $M$, with bivector field $\Lambda$ and vector field $E$. Since $\Lambda^n\wedge E\neq 0$ (where $\dim(M)=2n+1$) this Jacobi structure corresponds to a contact structure. In fact, the contact distribution is given by $\operatorname{span}\{\breve{\iota}_\alpha\Lambda\mid\alpha\in T*M\}$; and if $\theta\in\Omega^1(M)$ satisfies $\breve{\iota}_\theta\Lambda=0$ and $\breve{\iota}_\theta E=1$, then $\theta$ is a contact form. It is not difficult to see that this is the original contact structure, $\xi$, defined on $M$. (In fact, if $\omega=e^t(d\eta+dt\wedge\eta)$ is the symplectic form on $N$ (where $\eta\in\Ann(\xi)$ is nowhere vanishing), then $E$ is a reeb vector field for $\eta$ and $\theta=\eta$.)

We must note that, if instead of trivializing $V$ by the section $e^t\in\Gamma(V)$, we made the transformation $e^t\to fe^t$, for some nowhere vanishing $f\in C^\infty(M)$, then the appropriate changes to the Jacobi structure would be $\Lambda\to f\Lambda$, $E\to fE-\breve{\iota}_{df}\Lambda$, and the transformation for the contact form would be $\theta\to \frac{1}{f}\theta$. Thus it is clear that the freedom to modify the trivializing section of $V$ by a scalar multiple does not change the contact distribution, and fully accounts for the freedom to change the contact form by a scalar multiple. Indeed the generalized CR structure is defined intrinsically.
\end{example}

\section{Appendix: Proof of Proposition~\ref{mainprop}}
Suppose that $M$ is a manifold, $A$ is a Lie algebroid over $M$, $V$ is an $A$-module over $M$, and $\AV$ is an $AV$-Courant algebroid over $M$.

For $X,Y\in \Gamma(A)$, we have the following identities:

\begin{itemize}
\item $[\iota_X,\iota_Y]=0$
\item $[d,\iota_X]=\Lied_X$
\item $[\Lied_X,\iota_Y]=\iota_{[X,Y]}$
\item $[d,d]=0$
\item $[\Lied_X,d]=0$
\item $[\Lied_X,\Lied_Y]=\Lied_{[X,Y]}$
\end{itemize}

We will provide the proof we promised for Proposition~\ref{mainprop}, which we restate here:

\begin{proposition}
Let $\phi:A\to\AV$ be an isotropic splitting. Then under the isomorphism $\phi\oplus\jpi:A\oplus(V\otimes A^*)\to\AV$, the bracket is given by
$$\lb X+\xi,Y+\eta\rbp=[X,Y]+\Lied_X\eta - \iota_Y d\xi + \iota_X\iota_Y H_\phi,$$
where $X,Y\in \Gamma(A)$, $\xi,\eta\in\Gamma(V\otimes A^*)$, and $H_\phi\in\Gamma(V\otimes\wedge^3 A^*)$, with $ d H_\phi = 0$.

Furthermore, if $\psi:A\to\AV$ is a different choice of isotropic splitting, then $\psi(X)=\phi(X)+\jpi(\iota_X\beta)$, and $H_\psi=H_\phi- d\beta$, where $\beta\in\Gamma(V\otimes\wedge^2 A^*)$.
\end{proposition}
\begin{proof}
The proof will follow immediately from the following lemmas:
\end{proof}

\begin{lemma}
If $\xi\in \Gamma(V\otimes A^*)$, and $e\in \Gamma(\AV)$, then $\lb e,\jpi(\xi)\rb=\jpi(\Lied_{\pi(e)}\xi)$
\end{lemma}
\begin{proof}
Let $e_1,e_2\in\Gamma(\AV)$, $\xi\in \Gamma(V\otimes A^*)$.
$$\begin{array}{rcl}
\langle\lb e_1,\jpi(\xi)\rb,e_2\rangle &=&\Lied_{\pi(e_1)}\langle\jpi(\xi),e_2\rangle -\langle\jpi(\xi),\lb e_1,e_2\rb\rangle \\
&=& \Lied_{\pi(e_1)}\iota_{\pi(e_2)}\xi-\iota_{\pi([ e_1,e_2])}\xi\\
&=& \Lied_{\pi(e_1)}\iota_{\pi(e_2)}\xi-\iota_{[\pi( e_1),\pi(e_2)]}\xi\\
&=& \Lied_{\pi(e_1)}\iota_{\pi(e_2)}\xi-[\Lied_{\pi( e_1)},\iota_{\pi(e_2)}]\xi\\
&=& \iota_{\pi(e_2)}\Lied_{\pi(e_1)}\xi\\
&=& \langle \jpi(\Lied_{\pi(e_1)}\xi),e_2\rangle\\
\end{array}$$
\end{proof}

\begin{lemma}
If $\xi\in \Gamma(V\otimes A^*)$, and $e\in \Gamma(\AV)$, then $\lb \jpi(\xi),e\rb=-\jpi(\iota_{\pi(e)} d\xi)$.
\end{lemma}
\begin{proof}
$$\begin{array}{rcl}
\lb \jpi(\xi),e\rb&=&D\langle\jpi(\xi),e\rangle-\lb e,\jpi(\xi)\rb \\
&=&\jpi(d\iota_{\pi(e)}\xi)-\jpi(\Lied_{\pi(e)}\xi)\\
&=&\jpi( d\iota_{\pi(e)}\xi-(\iota_{\pi(e)} d\xi+ d\iota_{\pi(e)}\xi))\\
&=&-\jpi(\iota_{\pi(e)} d\xi)
\end{array}$$
\end{proof}

\begin{lemma}
If $\phi:A\to\AV$ is an isotropic splitting, and if $X,Y\in \Gamma(A)$ then 
$$\lb \phi(X),\phi(Y)\rb-\phi([X,Y])=\jpi(\iota_X\iota_YH),$$
where $H\in\Gamma(V\otimes\wedge^3 A^*)$.
\end{lemma}
\begin{proof}
Let $\phi$ be an isotropic splitting, and $X,Y,Z\in \Gamma(A)$. Then
$$\pi(\lb \phi(X),\phi(Y)\rb-\phi([X,Y]))=0,$$
so by exactness of the sequence (\ref{exactsequence}), $\lb \phi(X),\phi(Y)\rb-\phi([X,Y])\in \jpi(\Gamma(V\otimes A^*))$. We define $H$ by
$$\begin{array}{rcl}
H(X,Y,Z)&=&\langle\phi(Z),\lb \phi(X),\phi(Y)\rb-\phi([X,Y])\rangle \\
&=&\langle\phi(Z),\lb \phi(X),\phi(Y)\rb\rangle \\
\end{array},$$
where the second equality follows since $\phi$ is an isotropic splitting. It is obvious that $H$ is tensorial in $Z$. Furthermore, making repeated use of the fact that $\phi$ is an isotropic splitting, we check that $H$ is skew-symmetric:
$$\begin{array}{rcl}
\langle\phi(Z),\lb \phi(X),\phi(Y)\rb\rangle &=&\langle\phi(Z),-\lb\phi(Y),\phi(X)\rb+D\langle\phi(X),\phi(Y)\rangle \rangle \\
&=&-\langle\phi(Z),\lb\phi(Y),\phi(X)\rb\rangle \\
\end{array}$$
and
$$\begin{array}{rcl}
0&=&\Lied_X\langle\phi(Z),\phi(Y)\rangle \\
&=&\langle\lb\phi(X),\phi(Z)\rb,\phi(Y)\rangle +\langle\phi(Z),\lb\phi(X),\phi(Y)\rb\rangle \\
\end{array}$$
It follows that $H\in\Gamma(V\otimes\wedge^3 A^*)$.
\end{proof}

\begin{lemma}
Using the notation of the previous lemmas, $ d H=0$.
\end{lemma}
\begin{proof}
Using the fact that $[\Lied_X,\iota_Y]=\iota_{[X,Y]}$, it is easy to show that
$$
 d\iota_Z\iota_Y\iota_X+\iota_Z\iota_Y\iota_X d
=\Lied_Z\iota_Y\iota_X+\Lied_Y\iota_X\iota_Z+\Lied_X\iota_Z\iota_Y
+\iota_Z\iota_{[Y,X]}+\iota_Y\iota_{[X,Z]}+\iota_X\iota_{[Z,Y]}.
$$
Let $\phi:A\to\AV$ be an isotropic splitting. We shall use the identification 
$$A\oplus(V\otimes A^*) \xrightarrow{ \phi\oplus\jpi}\AV$$
 explicitly throughout this section. We have, for $X,Y,Z\in \Gamma(A)$,
$$\lb X,Y\rbp=[X,Y]+\iota_X\iota_YH.$$
Then using Axiom~(AV-\ref{AX4}) from the definition of an $AV$-Courant algebroid, we see that
$$\begin{array}{rcl}
0&=&\lb Z,\lb Y,X\rbp\rbp-\lb\lb Z,Y\rbp,X\rbp-\lb Y,\lb Z,X\rbp\rbp\\
&=&\lb Z,[Y,X]+\iota_Y\iota_XH\rbp
-\lb[Z,Y]+\iota_Z\iota_YH,X\rbp-\lb Y,[Z,X]+\iota_Z\iota_XH\rbp\\

&=&\lb Z,[Y,X]\rbp+\Lied_Z\iota_Y\iota_XH-\lb[Z,Y],X\rbp
+\iota_X d\iota_Z\iota_YH-\lb Y,[Z,X]\rbp-\Lied_Y\iota_Z\iota_XH\\

&=&\lb Z,[Y,X]\rbp-\lb[Z,Y],X\rbp-\lb Y,[Z,X]\rbp\\
&&+\Lied_Z\iota_Y\iota_XH
+\Lied_X\iota_Z\iota_YH+\Lied_Y\iota_X\iota_ZH- d\iota_Z\iota_Y\iota_XH\\

&=&[Z,[Y,X]]+\iota_Z\iota_{[Y,X]}H-[[Z,Y],X]-\iota_{[Z,Y]}\iota_XH-[Y,[Z,X]]-\iota_Y\iota_{[Z,X]}H\\
&&+\Lied_Z\iota_Y\iota_XH
+\Lied_X\iota_Z\iota_YH+\Lied_Y\iota_X\iota_ZH- d\iota_Z\iota_Y\iota_XH\\

&=&[Z,[Y,X]]-[[Z,Y],X]-[Y,[Z,X]]
+\iota_Z\iota_{[Y,X]}H+\iota_X\iota_{[Z,Y]}H+\iota_Y\iota_{[X,Z]}H\\
&&+\Lied_Z\iota_Y\iota_XH+\Lied_X\iota_Z\iota_YH+\Lied_Y\iota_X\iota_ZH
- d\iota_Z\iota_Y\iota_XH\\

&=&\iota_Z\iota_Y\iota_X d H.
\end{array}$$
\end{proof}

\begin{lemma}
Let $\phi:A\to\AV$ and $\psi:A\to\AV$ be two isotropic splittings, and let $H_\phi$ and $H_\psi$ be the elements of $\Gamma(V\otimes\wedge^3 A^*)$ associated to the corresponding splittings. Namely, if $X,Y\in \Gamma(A)$, then $\lb \phi(X),\phi(Y)\rb=\phi([X,Y])+\jpi\iota_X\iota_YH_\phi$, and similarly for $H_\psi$.

Then there exists $\beta\in\Gamma(V\otimes\wedge^2 A^*)$ such that $\psi(X)=\phi(X)+\jpi(\iota_X\beta)$ and $H_\psi=H_\phi- d\beta$.
\end{lemma}
\begin{proof}
Since $\phi$ and $\psi$ are splittings, we see that
$$\pi((\phi-\psi)(X))=0.$$
Thus, by the exactness of the sequence (\ref{exactsequence}), $(\phi-\psi)(X)=\jpi\circ S(X)$ for some linear map $S:A\to V\otimes A^*$.

However since the splittings are isotropic,
$$\begin{array}{rcl}
0&=&\langle\phi(X),\phi(Y)\rangle \\
&=&\langle\psi(X)+\jpi\circ S(X),\psi(Y)+\jpi\circ S(Y)\rangle\\
&=&S(X)(Y)+S(Y)(X),\\
\end{array}$$
so we can define $\beta\in\Gamma(V\otimes\wedge^2 A^*)$ by $\iota_X\beta=S(X)$. Then, we see that
$$\begin{array}{rcl}
\psi([X,Y])+\iota_X\iota_YH_\psi&=&\lb\phi(X)+\jpi(\iota_X\beta),\phi(Y)+\jpi(\iota_Y\beta)\rb\\
&=&\phi([X,Y])+\jpi(\Lied_X\iota_Y\beta-\iota_Y d\iota_X\beta+\iota_X\iota_YH_\phi)\\
&=&\phi([X,Y])+\jpi(\iota_X\iota_YH_\phi)
+\jpi(\Lied_X\iota_Y\beta-\iota_Y\Lied_X\beta+\iota_Y\iota_X d\beta)\\
&=&\phi([X,Y])+\jpi(\iota_X\iota_YH_\phi)
+\jpi(\iota_{[X,Y]}\beta+\iota_Y\iota_X d\beta)\\
&=&\phi([X,Y])+\jpi(\iota_{[X,Y]}\beta)
+\jpi(\iota_X\iota_YH_\phi-\iota_X\iota_Y d\beta)\\
&=&\psi([X,Y])+\jpi(\iota_X\iota_YH_\phi-\iota_X\iota_Y d\beta),\\
\end{array}$$
so we have
$$H_\psi=H_\phi- d\beta.$$
\end{proof}

\end{document}